\documentclass[a4paper,10pt]{amsart}

\usepackage{amsmath}
\usepackage{amsthm}
\usepackage{amssymb}
\usepackage{amsfonts}
\usepackage{mathrsfs}  
\usepackage{enumitem} 
\usepackage{xcolor}
\usepackage[bookmarks=false,hyperindex,pdftex,colorlinks,citecolor=blue,urlcolor=cyan]{hyperref}
\usepackage[a4paper,lmargin=3cm,rmargin=3cm,tmargin=4cm,bmargin=4cm,marginparwidth=2.8cm,marginparsep=1mm]{geometry} 
\usepackage{soul}
\usepackage[all]{xy} 

\DeclareMathOperator{\lspan}{span}                          
\DeclareMathOperator{\conv}{conv}                           
\DeclareMathOperator{\Lip}{Lip}                             
\DeclareMathOperator{\lip}{lip}                             

\newcommand{\NN}{\mathbb{N}}                                
\newcommand{\QQ}{\mathbb{Q}}                                
\newcommand{\RR}{\mathbb{R}}                                
\newcommand{\CC}{\mathbb{C}}                                
\newcommand{\HH}{\mathcal{H}}                               
\newcommand{\ep}{\varepsilon}

\newcommand{\abs}[1]{\left|{#1}\right|}                     
\newcommand{\pare}[1]{\left({#1}\right)}                    
\newcommand{\set}[1]{\left\{{#1}\right\}}                   
\newcommand{\norm}[1]{\left\|{#1}\right\|}                  
\newcommand{\duality}[1]{\left<{#1}\right>}                 
\newcommand{\cl}[1]{\overline{#1}}                          
\newcommand{\wsconv}{\stackrel{w^*}{\longrightarrow}}       
\newcommand{\restrict}{\mathord{\upharpoonright}}           
\newcommand{\lipfree}[1]{\mathcal{F}({#1})}                 

\renewcommand{\leq}{\leqslant}
\renewcommand{\geq}{\geqslant}

\theoremstyle{plain}
\newtheorem{theorem}{Theorem}[section]
\newtheorem{lemma}[theorem]{Lemma}
\newtheorem{corollary}[theorem]{Corollary}
\newtheorem{proposition}[theorem]{Proposition}

\newtheorem*{claim*}{Claim}

\newtheorem{question}[theorem]{Question}

\theoremstyle{definition}
\newtheorem*{definition*}{Definition}
\newtheorem{definition}[theorem]{Definition}
\newtheorem{example}[theorem]{Example}

\theoremstyle{remark}
\newtheorem{remark}[theorem]{Remark}

\begin{document}
\title{Lipschitz-free spaces and purely 1-unrectifiable metric spaces}

\author[R. J. Aliaga]{Ram\'on J. Aliaga}
\address[R. J. Aliaga]{Instituto Universitario de Matem\'atica Pura y Aplicada,
Universitat Polit\`ecnica de Val\`encia,
Camino de Vera S/N,
46022 Valencia, Spain}
\email{ramon.aliaga@upv.es}

\begin{abstract}
The Lipschitz-free space $\mathcal{F}(M)$ is a canonical linearization of a complete metric space $M$ whose topological dual is the space of Lipschitz functions on $M$. We review the properties of $\mathcal{F}(M)$ when the underlying space $M$ is purely 1-unrectifiable, that is, it contains no bi-Lipschitz copy of a subset of $\mathbb{R}$ with positive measure. For compact $M$, this is equivalent to several Banach space properties of $\mathcal{F}(M)$, including the Radon-Nikod\'ym and Schur properties or admitting a predual. We shall see how the study of locally flat Lipschitz functions on $M$ reveals these equivalences, and describe a technique that allows most of them to be transferred to the non-compact setting.

This manuscript is an expository text based on results by the author in collaboration with C. Gartland, C. Petitjean and A. Proch\'azka, originally published in a Trans. Amer. Math. Soc. paper,
and corresponds to a lecture delivered at the Second Winter School in Geometric Measure Theory at Westlake University, Hangzhou, on February 2026.
\end{abstract}



\maketitle

\section{Introduction}

Given a metric space $M$, one can construct a canonical linearization of $M$ as a Banach space $\lipfree{M}$ with the property that Lipschitz maps involving $M$ extend to bounded linear operators involving $\lipfree{M}$. This object is called the \emph{Lipschitz-free space over $M$} and, as it turns out, its topological dual is the space $\Lip_0(M)$ of real-valued Lipschitz functions on $M$ vanishing at a prespecified point. The fundamental papers in Lipschitz-free space theory are \cite{GodefroyKalton,Kalton} by Godefroy and Kalton, and the main resource for reference is Weaver's monograph \cite{Weaver2} (although the latter tends to focus on $\Lip_0(M)$ rather than $\lipfree{M}$).

One of the main draws of the theory of Lipschitz-free spaces is precisely that it provides linearizations of metric spaces and Lipschitz mappings between them, thus allowing us to study their metric and Lipschitz structure using linear tools from Banach space theory. This linearization procedure simplifies the mappings (Lipschitz functions become linear) at the cost of making the spaces themselves more complicated - the structure of $\lipfree{M}$ is, in general, notoriously more intricate than that of $M$. Thus, any information about the structure of $\lipfree{M}$ can potentially be used to extract non-trivial metric consequences. That was already exploited in the seminal paper \cite{GodefroyKalton}: by considering Lipschitz-free spaces over Banach spaces (regarded as metric spaces with their norm metric), Godefroy and Kalton showed that the bounded approximation property is preserved between bi-Lipschitz equivalent Banach spaces. Lipschitz-free spaces also appear naturally in other areas, such as in optimal transport theory as $1$-Wasserstein spaces (see \cite{Villani} and \cite[Section 1.2]{APS1}) and in geometric measure theory as certain spaces of currents (see the recent preprints \cite{BCTVW,DePauw}).

A common theme in Lipschitz-free space theory is finding implications (equivalences, if possible) between metric properties of $M$ and functional analytic properties of $\lipfree{M}$. For example, previous works have established metric characterizations of Banach space properties of $\lipfree{M}$ such as the Daugavet property \cite{GPR}, octahedrality \cite{PRZ}, or being isometric to (a subspace of) the sequence space $\ell_1$ \cite{APP,DKP}; and, conversely, functional analytic characterizations of metric properties of $M$ such as being geodesic \cite{APS3} or $0$-hyperbolic \cite{Godard}.

This paper is devoted to another such equivalence, that unexpectedly links several well-established, important properties from both fields. On the metric side, we consider the following notion.

\begin{definition}
\label{def:p1u}
A metric space $M$ is \emph{purely $1$-unrectifiable} if $\HH^1(\varphi(A))=0$ for every subset $A\subset\RR$ and every Lipschitz function $\varphi:A\to M$.
\end{definition}

Here and hereafter, $\HH^1$ denotes $1$-dimensional Hausdorff measure. As a simple consequence of Kirchheim's work on the metric differential \cite{Kirchheim}, one may replace ``Lipschitz'' with ``bi-Lipschitz'' in Definition \ref{def:p1u}. This yields the following alternative characterization of pure $1$-unrectifiability in terms of \emph{curve fragments}, that is, bi-Lipschitz copies of compact subsets $K\subset\RR$ with $\HH^1(K)>0$.

\begin{theorem}
\label{th:p1u kirchheim}
A metric space is purely $1$-unrectifiable if and only if it contains no curve fragment.
\end{theorem}

\noindent See e.g. \cite[Section 1.3]{AGPP} for details of the proof of Theorem \ref{th:p1u kirchheim}.

\medskip
On the Banach space side, we consider the following properties.

\begin{definition}
\label{def:rnp}
A Banach space $X$...
\begin{enumerate}[label={\upshape{(\alph*)}}]
\item ...has the \emph{Radon-Nikod\'ym property} if every Lipschitz mapping $f:[0,1]\to X$ is differentiable almost everywhere.
\item ...has the \emph{Schur property} if every sequence in $X$ that converges in the weak topology also converges in norm.
\end{enumerate}
\end{definition}

The Radon-Nikod\'ym property, in particular, is arguably the single most important property in Banach space theory and has \textit{many} equivalent characterizations, so different texts will often give different definitions for it. For instance, $X$ has the Radon-Nikod\'ym property if and only if...
\begin{itemize}
\item ...the Radon-Nikod\'ym theorem holds for $X$-valued vector measures (this is, of course, the original definition and name giver).
\item ...every closed, bounded, convex set $C\subset X$ is the closed convex hull of its denting points.
\end{itemize}
See Section VII.6 in \cite{DiestelUhl} for a more exhaustive list of equivalent formulations of the Radon-Nikod\'ym property. An additional one will be given in Section \ref{sec:probabilistic}.

\medskip
The main result presented here is the correspondence between the properties in Definitions \ref{def:p1u} and \ref{def:rnp}. Namely, the Radon-Nikod\'ym and Schur properties are equivalent for Lipschitz-free spaces $\lipfree{M}$, and they hold if and only if the completion of $M$ is purely $1$-unrectifiable. See Theorem \ref{th:non-compact} below for additional equivalent conditions, and Theorem \ref{th:compact} for further conditions that are equivalent only in the case where $M$ is compact. Notably, both the metric property (pure $1$-unrectifiability) and the functional analytic properties (Radon-Nikod\'ym and Schur) are non-trivial, highly specialized properties that are seldom used outside of their respective areas (geometric measure theory and Banach space theory, respectively). Thus, this result establishes an unexpected bridge between both theories, and we can only hope that more such bridges can be found in the near future.

The equivalence theorem was obtained in \cite{AGPP} by the author of this text, together with C. Gartland, C. Petitjean and A. Proch\'azka, but it builds on many previous intermediate and related results by many authors scattered throughout different references including \cite{ANPP,Bate,Godard,Kalton,Weaver2}. This paper is intended as an exposition of the equivalence result and its proof but, given its rather large scope, we will not attempt to provide rigorous arguments for every single step. Instead, we will sketch the main ideas behind some of the more lengthy constructions, and only provide full proofs of some of the smaller steps. Full references are provided, and we encourage the interested reader to consult them in order to find the missing details.

The paper is structured as follows. In Section \ref{sec:locally flat} we analyze the purely metric side of the problem, namely, the relation between pure $1$-unrectifiability and the possibility of approximating Lipschitz functions on $M$ with locally flat functions. Besides being one of the core ideas behind the main result, this is a problem of independent interest that is moreover still open in full generality. The functional analytic content of the paper starts in Section \ref{sec:lipfree}, where Lipschitz-free spaces are formally introduced and their main properties are described and proved. Finally, the last two sections are devoted to the proof of the main result. We first prove equivalence for compact $M$ in Section \ref{sec:compact}, and then in Section \ref{sec:non-compact} we describe a general technique that allows this and other problems to be reduced to the compact case.

\section{Locally flat Lipschitz functions}
\label{sec:locally flat}

Throughout this document, $M$ will denote a metric space with metric $d$, and the closed ball with center $x\in M$ and radius $r>0$ will be denoted $B(x,r)$. Much of our attention will be directed towards real-valued functions $f:M\to\RR$. For any such function, its (optimal) Lipschitz constant is
$$
\Lip(f) = \sup\set{\frac{f(x)-f(y)}{d(x,y)} \,:\, x\neq y\in M} ,
$$
and we say that $f$ is Lipschitz if $\Lip(f)<\infty$. We denote by
$$
\Lip(M) = \set{f:M\to\RR \,:\, \text{$f$ is Lipschitz}}
$$
the set of all Lipschitz functions $f:M\to\RR$. One easily checks that $\Lip(f+g)\leq\Lip(f)+\Lip(g)$ and $\Lip(\max\set{f,g})\leq\max\set{\Lip(f),\Lip(g)}$, thus $\Lip(M)$ is a vector space that is closed under taking finite minima or maxima. McShane's well-known extension theorem (see e.g. \cite[Theorem 1.33]{Weaver2}) states that any $f\in\Lip(N)$ defined on a subset $N\subset M$ can be extended to a function $F\in\Lip(M)$ (in the sense that $F\restrict_N=f$) without increasing its Lipschitz constant, that is, so that $\Lip(F)=\Lip(f)$.

In this section, we focus on a special class of Lipschitz functions. Many of the arguments used here to work with them can be found explicitly or implicitly in Chapter 4 of Weaver's book \cite{Weaver2}.

\begin{definition}\label{def:locally flat}
Let $M$ be a metric space, $f:M\to\RR$ a function, and fix $x\in M$.
\begin{enumerate}[label={\upshape{(\alph*)}}]
\item The \emph{asymptotic Lipschitz constant} of $f$ at $x$ is the value
\begin{equation}\label{eq:asymptotic_lip}
\lim_{r\to 0}\Lip(f\restrict_{B(x,r)}) = \lim_{r\to 0}\sup_{\substack{{y,z\in B(x,r)}\\{y\neq z}}}\frac{\abs{f(y)-f(z)}}{d(y,z)} .
\end{equation}
The supremum is understood to be $0$ if there are no such $y,z$. That is, the asymptotic Lipschitz constant is $0$ at isolated points.
\item We say that $f$ is \emph{locally flat at $x$} if its asymptotic Lipschitz constant at $x$ is $0$. If $f$ is locally flat at every $x\in M$, then we simply say that $f$ is \emph{locally flat}.
\end{enumerate}
\end{definition}

Local flatness at a given point $x\in M$ can be understood as a purely metric version of ``having null gradient at $x$''. A uniform version exists as well: we say that $f$ is \emph{uniformly locally flat} if the limit \eqref{eq:asymptotic_lip} is 0 uniformly on $x\in M$, that is, if
$$
\lim_{r\to 0}\sup_{\substack{{y,z\in M}\\{0<d(y,z)\leq r}}}\frac{\abs{f(y)-f(z)}}{d(y,z)} = 0 .
$$
A standard compactness argument shows that every locally flat function is uniformly locally flat when $M$ is compact, but this is no longer true for general $M$.

The set of all functions in $\Lip(M)$ that are locally flat is denoted by
$$
\lip(M) = \set{f\in\Lip(M) \,:\, \text{$f$ is locally flat}} .
$$
The following properties of locally flat functions are easy to check:
\begin{itemize}
\item $\lip(M)$ is a vector space.
\item The pointwise maximum and minimum of two functions in $\lip(M)$ also belong to $\lip(M)$.
\item Composing a locally flat function with a Lipschitz function (in any order) yields a locally flat function.
\item Given a function $f$ and a sequence $(f_n)$ in $\Lip(M)$ such that $\Lip(f-f_n)\to 0$, if all $f_n$ are locally flat then $f$ is locally flat as well.
\end{itemize}

Given a subset $N\subset M$, a function in $\lip(N)$ cannot in general be extended to a function in $\lip(M)$ (see \cite[Section 4.4]{Weaver2} for some situations where they can), but extensions with weaker properties are possible. An improvement of McShane's theorem due to Di Marino, Gigli and Pratelli \cite{DGP} states that any $f\in\Lip(N)$ can be extended to a function $F\in\Lip(M)$, with $\Lip(F)\leq\Lip(f)+\ep$ for any given $\ep>0$, in such a way that the asymptotic Lipschitz constants of $f$ and $F$ agree on every point of $N$. Unlike McShane's theorem, this result does not hold with $\ep=0$ in general. As a particular case of it, we have:

\begin{theorem}[cf.~\cite{DGP}]\label{th:locally_flat_extension}
Let $M$ be a metric space, $N\subset M$ and $\ep>0$. Then any $f\in\lip(N)$ admits an extension $F\in\Lip(M)$, with $\Lip(F)\leq\Lip(f)+\ep$, that is locally flat at every point of $N$.
\end{theorem}

It is instructive to analyze how the space of locally flat Lipschitz functions looks like for different subsets of $\RR$.

\begin{example}\label{ex:examples_lip}
\hfill
\begin{enumerate}[label={\upshape{(\alph*)}}]
\item \label{ex:lipR} Suppose first that $M=[0,1]$. If $f:M\to\RR$ is locally flat at a certain $x\in M$, then in particular $f'(x)=0$. Thus, if $f\in\lip(M)$ then the fundamental theorem of calculus implies that $f$ is constant. The same conclusion holds if we take $M=\RR$.
\item Now take $M=\NN$. Since every point of $M$ is isolated, every Lipschitz function on $M$ is locally flat at every point. Thus $\lip(M)=\Lip(M)$.
\item \label{ex:real_null} A more interesting situation arises if we look at an intermediate case. Suppose that $M$ is a (not necessarily closed) subset of $\RR$ with $\HH^1(M)=0$. If $M$ has non-isolated points then it is no longer true that $\lip(M)=\Lip(M)$, as e.g. the identity $f(x)=x$ has asymptotic Lipschitz constant equal to $1$ at every such point. However, there are plenty of locally flat functions; in fact, every Lipschitz function on $M$ can be approximated pointwise by locally flat functions.

To see this, fix $\ep>0$, and cover $M$ with an open set $U\subset\RR$ such that $\HH^1(U)<\ep$. Now define $f:\RR\to\RR$ by
$$
f(x) = \int_0^x (1-\chi_U(t))\,dt
$$
where $\chi_U$ denotes the characteristic function of $U$. That is, $f$ stays constant within connected components of $U$, and increases with slope $1$ outside of them. It is clear that $f$ is $1$-Lipschitz. While the function $f$ itself is not locally flat, note that the restriction $f\restrict_M$ is. Indeed, $f\restrict_M$ is even locally constant, as every point $x\in M$ is contained in an open connected component of $U$ on which $f$ is constant. Hence $f\restrict_M\in\lip(M)$. Moreover, for any pair of points $x,y\in M$ with $x<y$, we have
$$
f(y)-f(x) = \int_x^y (1-\chi_U(t))\,dt = \HH^1([x,y]\setminus U) > (y-x)-\ep .
$$
Thus, although $\lip(M)$ does not contain \textit{all} Lipschitz functions, it contains enough of them to be able to separate pairs of points arbitrarily well. As we will see in Proposition \ref{pr:spu eqv density} below, this is equivalent to the statement that functions in $\Lip(M)$ can be approximated, in the topology of pointwise convergence, by locally flat functions with the same Lipschitz constant.
\end{enumerate}
\end{example}

The separation property obtained in Example \ref{ex:examples_lip}\ref{ex:real_null} was first isolated by Weaver in \cite{Weaver96_2}. Originally, it was simply called ``separation property'', but the following term is much more common now.

\begin{definition}\label{def:spu}
Let $M$ be a metric space. We say that $\lip(M)$ \emph{separates points of $M$ uniformly} if the following holds: given any $x,y\in M$ and any $\ep>0$, there exists a function $f\in\lip(M)$ with $\Lip(f)\leq 1$ such that $f(y)-f(x) > d(x,y)-\ep$ (or equivalently, there exists a function $f\in\lip(M)$ with $\Lip(f)<1+\ep$ such that $f(y)-f(x)=d(x,y)$).
\end{definition}

(In the notation of \cite{Weaver2}, we would say that $\lip(M)$ separates points of $M$ uniformly \emph{with factor/constant $1$}, but we will not make such fine distinctions.)

As foreshadowed above, this separation property can be restated as an approximation property as follows.

\begin{proposition}\label{pr:spu eqv density}
Let $M$ be a metric space. Then the following are equivalent:
\begin{enumerate}[label={\upshape{(\roman*)}}]
\item $\lip(M)$ separates points of $M$ uniformly,
\item the set $\set{f\in\lip(M) : \Lip(f)\leq 1}$ is dense in $\set{f\in\Lip(M) : \Lip(f)\leq 1}$ in the topology of pointwise convergence.
\end{enumerate}
\end{proposition}

\begin{proof}
(ii)$\Rightarrow$(i): Fix $x\neq y\in M$ and consider the $1$-Lipschitz function $f(p)=d(x,p)$. By (ii), for every $\ep>0$ there exists $g\in\lip(M)$ with $\Lip(g)\leq 1$ such that $\abs{f(x)-g(x)}<\ep/2$ and $\abs{f(y)-g(y)}<\ep/2$, therefore $g(y)-g(x)>f(y)-f(x)-\ep=d(x,y)-\ep$ as required.

(i)$\Rightarrow$(ii): Let $f\in\Lip(M)$ be $1$-Lipschitz, $x_1,\ldots,x_n$ be distinct points in $M$, and $\ep>0$. We wish to find $g\in\lip(M)$ with $\Lip(g)\leq 1$ such that $\abs{f(x_i)-g(x_i)}\leq\ep$ for all $i$. Fix $C>\max_{1\leq i\leq n}\abs{f(x_i)}$ and set $\eta=\ep/C$. By (i), for every pair $1\leq i,j\leq n$, $i\neq j$ there exists a function $f_{ij}\in\lip(M)$ with $\Lip(f_{ij})<1+\eta$ and $f_{ij}(x_i)-f_{ij}(x_j)=d(x_i,x_j)$. Setting
$$
h_{ij}(x) = \frac{f(x_i)-f(x_j)}{d(x_i,x_j)}\cdot (f_{ij}(x)-f_{ij}(x_i))+f(x_i)
$$
yields $h_{ij}\in\lip(M)$ with $\Lip(h_{ij})\leq\Lip(f_{ij})<1+\eta$, $h_{ij}(x_i)=f(x_i)$, and $h_{ij}(x_j)=f(x_j)$. Now let
$$
h(x) = \max_{1\leq i\leq n} \min_{\substack{1\leq j\leq n \\ j\neq i}} h_{ij}(x) .
$$
Since $\lip(M)$ is closed under taking pairwise maxima and minima, $h$ is locally flat and $\Lip(h)<1+\eta$. Moreover, $h(x_i)=f(x_i)$ for all $1\leq i\leq n$. Finally, let $g=h/(1+\eta)$. Then $g$ is locally flat, $1$-Lipschitz, and
$$
\abs{f(x_i)-g(x_i)} = \pare{1-\frac{1}{1+\eta}}\abs{f(x_i)} \leq \eta\abs{f(x_i)} \leq \ep
$$
as needed.
\end{proof}

The separation property turns out to be crucial for the set of equivalences presented in this text; the link will be provided by Theorem \ref{th:spu dual} below. But even without that context, its reformulation in terms of pointwise density makes it natural to ask for which metric spaces it holds. A (rather strong) necessary condition is given by pure unrectifiability.

\begin{proposition}\label{pr:spu p1u}
Let $M$ be a metric space. If $\lip(M)$ separates points of $M$ uniformly, then $M$ is purely $1$-unrectifiable.
\end{proposition}

\begin{proof}
Suppose that $M$ is not purely $1$-unrectifiable, that is, it contains a curve fragment. We will show that there exist $x\neq y\in M$ such that $f(y)-f(x)\leq\frac{1}{2}d(x,y)$ for every $1$-Lipschitz, locally flat $f:M\to\RR$. It then follows that $\lip(M)$ does not separate points of $M$ uniformly.

By assumption there exist a subset $K\subset\RR$ with positive measure and a bi-Lipschitz embedding $\varphi:K\to M$. Let $D=\Lip(\varphi)\Lip(\varphi^{-1})$ be the bi-Lipschitz distortion of $\varphi$. By Lebesgue's density theorem, we may find $p<q$ in $K$ such that $\HH^1([p,q]\setminus K)\leq\frac{1}{4D}(q-p)$. Let us see that $x=\varphi(p)$, $y=\varphi(q)$ will do. Indeed, let $f\in\lip(M)$ be $1$-Lipschitz, and set $g=f\circ\varphi\in\lip(K)$. By Theorem \ref{th:locally_flat_extension}, $g$ can be extended to a function $G\in\Lip(\RR)$ that is locally flat at every point of $K$ and moreover has Lipschitz constant $\Lip(G)\leq 2\Lip(g)\leq 2\Lip(\varphi)$. Then
$$
f(y)-f(x) = G(q)-G(p) = \int_p^q G'(t)\,dt
$$
by the fundamental theorem of calculus. But $G'(t)=0$ for every $t\in K$, so
$$
f(y)-f(x) \leq \Lip(G)\cdot\HH^1([p,q]\setminus K) \leq \Lip(G)\frac{q-p}{4D} \leq 2\Lip(\varphi)\frac{\Lip(\varphi^{-1})d(x,y)}{4D} = \frac{1}{2}d(x,y) .
$$
This ends the proof.
\end{proof}

Whether the converse of Proposition \ref{pr:spu p1u} holds is, to the best of our knowledge, an open question at the time of writing.

\begin{question}\label{q:spu p1u}
Let $M$ be a purely $1$-unrectifiable metric space. Does $\lip(M)$ separate points of $M$ uniformly?
\end{question}

Question \ref{q:spu p1u} is known to have a positive answer in some particular cases. For instance, it holds when $M$ is any metric space with null $1$-Hausdorff measure (note that any such $M$ trivially satisfies the definition of pure $1$-unrectifiability).

\begin{proposition}
Let $M$ be a metric space. If $\HH^1(M)=0$, then $\lip(M)$ separates points of $M$ uniformly.
\end{proposition}

\begin{proof}
Fix points $x,y\in M$ and $\ep>0$. Let $g\in\Lip(M)$ denote the $1$-Lipschitz function given by $g(p)=d(x,p)$. Then the set $g(M)\subset\RR$ satisfies $\HH^1(g(M)) \leq \Lip(g)\cdot\HH^1(M) = 0$. Therefore, the construction from Example \ref{ex:examples_lip}\ref{ex:real_null} shows that $\lip(g(M))$ separates points of $g(M)$ uniformly. Thus we may find a $1$-Lipschitz, locally flat function $f\in\lip(g(M))$ such that
$$
f(g(y))-f(g(x)) > \abs{g(y)-g(x)}-\ep = d(x,y)-\ep .
$$
So the $1$-Lipschitz, locally flat function $f\circ g\in\lip(M)$ separates $x$ and $y$ by their full distance up to $\ep$, as needed.
\end{proof}

Another particular situation where Question \ref{q:spu p1u} is known to have a positive answer is that of snowflake metrics. Recall that, given a metric space $M$ and a real number $\alpha\in (0,1)$, the \emph{snowflake} $M^\alpha$ is the metric space $(M,d_{M^\alpha})$ endowed with the metric $d_{M^\alpha}=d^\alpha$, that is, $d_{M^\alpha}(x,y)=d(x,y)^\alpha$ for $x,y\in M$. It is readily checked that $d^\alpha$ is indeed a metric on $M$. It is also not too hard to verify directly that every snowflake is purely $1$-unrectifiable.

\begin{proposition}
Let $M$ be a metric space and $\alpha\in(0,1)$. Then $\lip(M^\alpha)$ separates points of $M^\alpha$ uniformly.
\end{proposition}

\begin{proof}
First note that every function $f\in\Lip(M)$ is uniformly locally flat with respect to the snowflake metric of $M^\alpha$: indeed,
$$
\frac{\abs{f(x)-f(y)}}{d_{M^\alpha}(x,y)} \leq \frac{\Lip(f)\cdot d(x,y)}{d(x,y)^\alpha} = \Lip(f)\cdot d(x,y)^{1-\alpha}
$$
converges uniformly to $0$ as $d(x,y)\to 0$.

Fix $x\neq y\in M$ and set
\begin{equation}\label{eq:snowflake_f}
f(p)=\frac{\min\set{d(x,p),d(x,y)}}{d(x,y)^{1-\alpha}}
\end{equation}
for $p\in M$. The numerator of \eqref{eq:snowflake_f} is clearly a $1$-Lipschitz function with respect to the original metric of $M$, so $f$ is (uniformly) locally flat as a function on $M^\alpha$. It is also clear that $f(x)=0$ and $f(y)=d(x,y)^\alpha$. Now, given $p,q\in M$, if $d(p,q)\leq d(x,y)$ then we have
$$
\abs{f(p)-f(q)} \leq \frac{d(p,q)}{d(x,y)^{1-\alpha}} \leq \frac{d(p,q)^\alpha d(x,y)^{1-\alpha}}{d(x,y)^{1-\alpha}} = d(p,q)^\alpha ,
$$
and if $d(p,q)\geq d(x,y)$ then, since the numerator of \eqref{eq:snowflake_f} takes values between $0$ and $d(x,y)$,
$$
\abs{f(p)-f(q)} \leq \frac{d(x,y)}{d(x,y)^{1-\alpha}} = d(x,y)^\alpha \leq d(p,q)^\alpha .
$$
Thus, in any case $\abs{f(p)-f(q)} \leq d(p,q)^\alpha$ and $f$ is $1$-Lipschitz as an element of $\lip(M^\alpha)$, while $f(y)-f(x)=d(x,y)^\alpha$. This ends the proof.
\end{proof}

The most involved partial answer to Question \ref{q:spu p1u} known to date is the compact case. Although it was first stated and established in \cite{AGPP} by the author, together with Gartland, Petitjean and Proch\'azka, it draws heavily from ideas in Bate's paper \cite{Bate} (and it can, indeed, be deduced in a relatively direct manner from some intermediate results in \cite{Bate}) and so one may say that it was essentially proved by Bate. This compact case will also be the crucial step in the arguments in the forthcoming Section \ref{sec:compact}.

\begin{theorem}[{\cite[Theorem A]{AGPP}}]
\label{th:agpp spu}
Let $M$ be a compact, purely $1$-unrectifiable metric space. Then $\lip(M)$ separates points of $M$ uniformly.
\end{theorem}

We will not attempt to provide a full, rigorous proof of Theorem \ref{th:agpp spu} here; we refer the interested reader to Section 2 of \cite{AGPP}. Instead, we will only sketch the main ideas and handwave the technical computations and difficulties off.

The construction behind Theorem \ref{th:agpp spu} is, at its core, just a suitable generalization of Example \ref{ex:examples_lip}\ref{ex:real_null}. In the example above, where $M$ is a purely $1$-unrectifiable subset of $\RR$, we define the desired separating function $f$ not just on $M$, but in its overspace $\RR$. In order to ensure that $f$ is locally flat at every point of $M$, it is constructed as the antiderivative of a function $g$ that vanishes in a neighborhood $U$ of $M$ -- namely, $g=\chi_{\RR\setminus U}$. By taking $U$ as small as possible, we guarantee that $g$ is large (equal to $1$) away from $M$, and this yields the required separation.

In order to transfer these ideas to a general compact metric setting, we proceed as follows. First, we take an overspace $\Omega$ of $M$ that is a compact convex subset of some Banach space; for instance, we embed $M$ isometrically into any Banach space ($\ell_\infty$ will do) and then consider its closed convex hull $\Omega=\cl{\conv}(M)$. Next, we define our ``indicator'' function $g$ in a slightly more general way than in Example \ref{ex:examples_lip}\ref{ex:real_null}: we consider a suitable nested sequence $(U_n)$ of neighborhoods of $M$, converging to $M$, and we let
$$
g(x) = \begin{cases}
1 &\text{, if $x\notin U_1$} \\
\frac{1}{n} &\text{, if $x\in U_n\setminus U_{n+1}$} \\
0 &\text{, if $x\in M$}
\end{cases} .
$$
Finally, in order to define the ``antiderivative'' $f$, we replace standard integration in $\RR$ with integration along curves. We fix one of the points $p$ we intend to separate, and let
\begin{equation}\label{eq:anti upper gradient}
f(x) = \inf\set{\int_\gamma g \,:\, \text{$\gamma$ is a Lipschitz curve in $\Omega$ starting at $p$ and ending at $x$}}
\end{equation}
so that $g$ is an upper gradient of $f$, rather than its derivative. One easily checks that $f$ is $1$-Lipschitz (because $\abs{g}\leq 1$ and $\Omega$ is convex) and that it is locally flat at every point of $M$ (because $g$ vanishes continuously at every point of $M$).

It only remains to make sure that $f(x)-f(p)>d(p,x)-\ep$ for a fixed (or for any given) $x\in M$ by choosing the sequence $(U_n)$ appropriately. We will only provide an intuitive reason why this should be possible. Since $M$ is purely $1$-unrectifiable, its intersection with any Lipschitz curve $\gamma$ in $\Omega$ is $\HH^1$-null. So it seems reasonable that we could choose neighborhoods $U_n$ that are very close to $M$ so that $\gamma\cap U_n$ has small measure. Then, for any curve $\gamma$ going from $p$ to $x$, most of $\gamma$ would be located in regions where $g$ is large (preferably outside of $U_1$), so that all integrals $\int_\gamma g$ considered in \eqref{eq:anti upper gradient} are large. Achieving this quantitatively is the main technical burden in the proof of Theorem \ref{th:agpp spu} and we shall omit the details here.

The construction sketched above is based on a similar construction found in Section 3 of Bate's paper \cite{Bate}, also for compact, purely $1$-unrectifiable $M$. Recall that, by Proposition \ref{pr:spu eqv density}, uniform separation of points by $\lip(M)$ is equivalent to approximability of $1$-Lipschitz functions by locally flat functions in the pointwise sense, which is equivalent to the uniform sense for compact $M$. Bate is instead concerned with approximability by functions whose asymptotic Lipschitz constants are uniformly small. Indeed, as a particular case of \cite[Lemma 3.4]{Bate} one has the following:

\begin{lemma}[{cf.~\cite[Lemma 3.4]{Bate}}]\label{lm:bate}
Let $M$ be a compact, purely $1$-unrectifiable metric space, $f\in\Lip(M)$ and $\ep>0$. Then there exist $g\in\Lip(M)$ and $\delta>0$ such that
\begin{align}
& \abs{g(x)-f(x)} \leq \ep \text{ for all $x\in M$,} \tag{i} \\
& \dfrac{\abs{g(x)-g(y)}}{d(x,y)} \leq \dfrac{\abs{f(x)-f(y)}}{d(x,y)} + \ep \text{ for all $x\neq y\in M$, and} \tag{ii} \\
& \dfrac{\abs{g(x)-g(y)}}{d(x,y)} \leq \ep \text{ for all $x\neq y\in M$ with $d(x,y)\leq\delta$.} \tag{iii}
\end{align}
\end{lemma}

Theorem \ref{th:agpp spu} can be proved directly by repeated application of Lemma \ref{lm:bate}. Indeed, set $g_0=f$ and apply the lemma iteratively to $f=g_{n-1}$ and $\ep=2^{-n}\ep$ to obtain a new function $g_n$. Then the sequence $(g_n)$ converges uniformly by (i), and the limit $g$ is a uniform approximation of $f$ that is $(1+\ep)$-Lipschitz by (ii) and locally flat by (iii). The full details of the argument can be found in \cite[Proposition 6.15]{FJLPPQ}. Note that Lemma \ref{lm:bate} is in fact \textit{stronger} than Theorem \ref{th:agpp spu}, as condition (ii) is stronger than simply $\Lip(g)\leq\Lip(f)+\ep$.

\section{Lipschitz-free spaces}
\label{sec:lipfree}

In this section, we introduce Lipschitz and Lipschitz-free spaces. Some of the content and arguments here are taken from \cite[Chapters 1 to 3]{Weaver2}, \cite[Section 2]{CDW}, and \cite[Chapter 2]{Aliaga_thesis}.

From this point onward, we consider metric spaces $M$ to be \emph{pointed}, meaning that we designate a certain point of $M$ as a \emph{base point} (the choice of base point is completely inconsequential). We will usually denote the base point of $M$ by $0$ or $0_M$. Then we consider the restricted space of Lipschitz functions
$$
\Lip_0(M) = \set{f\in\Lip(M) \,:\, f(0)=0} .
$$
This is usually called the \emph{Lipschitz space} over $M$. Adding the constraint $f(0)=0$ has the effect that the only function with $\Lip(f)=0$ is $f\equiv 0$ (while any constant function will do in $\Lip(M)$). It follows that $\Lip(\cdot)$ defines a norm on the Lipschitz space, which we will call the Lipschitz norm. In fact, it is a complete norm, and so $\Lip_0(M)$ is a Banach space. Lipschitz spaces can be understood as the metric counterpart of spaces $C(K)$ of continuous functions on a compact Hausdorff space for topology, or spaces $L^\infty(\Omega)$ of measurable functions for measure theory.

As is usual when dealing with function spaces, the simplest continuous linear functionals on $\Lip_0(M)$ (i.e.~elements of its topological dual $\Lip_0(M)^*$) are the evaluation operators $\delta(x)$ for $x\in M$, given by $\duality{f,\delta(x)}=f(x)$. This defines a mapping $\delta:M\to\Lip_0(M)^*$ that is an isometric embedding. Indeed,
$$
\norm{\delta(x)-\delta(y)} = \sup\set{\duality{f,\delta(x)-\delta(y)}:f\in B_{\Lip_0(M)}} = \sup\set{f(x)-f(y):f\in B_{\Lip_0(M)}}
$$
(here and hereafter, $B_X$ denotes the closed unit ball of a Banach space $X$ as is customary). The right-hand side is bounded by $d(x,y)$ by the definition of Lipschitz constant, and a trivial computation shows that this value is attained for a function of the form $d(\cdot,y)-d(0,y)$; thus, $\norm{\delta(x)-\delta(y)}=d(x,y)$. Of course, $\delta(0)=0$ is just the null functional as $f(0)=0$ for all $f\in\Lip_0(M)$.

\begin{definition}
The \emph{Lipschitz-free space} over $M$ is the closed subspace of $\Lip_0(M)^*$ given by
$$
\lipfree{M} = \cl{\lspan}\,\delta(M) = \cl{\lspan}\set{\delta(x)\,:\,x\in M} .
$$
\end{definition}

The Lipschitz-free space $\lipfree{M}$ can be regarded as a (\textit{the}) canonical linearization of the metric space $M$ in the following sense.

\begin{theorem}\label{th:lipfree_universal}
$\lipfree{M}$ is the unique (up to isometry) Banach space with these properties:
\begin{enumerate}[label={\upshape{(\alph*)}}]
\item \label{th:lipfree_universal:embedding} there is an isometric embedding $\delta:M\to\lipfree{M}$, with $\delta(0)=0$, such that $\delta(M)\setminus\set{0}$ is linearly dense and linearly independent, and
\item \label{th:lipfree_universal:linearization} given any Banach space $X$ and any Lipschitz function $f:M\to X$ with $f(0)=0$, there exists a (unique) bounded linear operator $\bar{f}:\lipfree{M}\to X$ that extends $f$ in the sense that $\bar{f}\circ\delta=f$, and $\norm{\bar{f}}=\Lip(f)$.
\end{enumerate}
\end{theorem}

Property \ref{th:lipfree_universal:linearization} can be summarized by means of the following commutative diagram:
$$
\xymatrix{
M \ar[d]_{\delta} \ar[r]^f & X \\
\lipfree{M} \ar@{.>}[ur]_{\bar{f}} &
}
$$

\begin{proof}
We have already shown that the evaluation map $\delta$ is an isometry and that $\delta(0)=0$. Given a finite set $E=\set{x_1,\ldots,x_n}\subset M\setminus\set{0}$, the linear independence of the $\delta(x_k)$ follows from considering the functions $f_k\in\Lip_0(M)$ given by $f_k(x)=d(x,\set{0}\cup E\setminus\set{x_k})$, as $f_k(x_j)\neq 0$ if and only if $k=j$. This establishes \ref{th:lipfree_universal:embedding}.

To prove \ref{th:lipfree_universal:linearization}, given $f:M\to X$ define $\bar{f}$ on $\delta(M)$ by $\bar{f}(\delta(x))=f(x)$ for $x\in M$, and then extend linearly to an operator $\bar{f}:\lspan\,\delta(M)\to X$ using linear independence. Fix an element $m=\sum_{k=1}^na_k\delta(x_k)$ of $\lspan\,\delta(M)$ and let $x^*\in B_{X^*}$ be such that $x^*(\bar{f}(m))=\norm{\bar{f}(m)}$. Note that $x^*\circ f\in\Lip_0(M)$ with $\Lip(x^*\circ f)\leq\Lip(f)$. Thus
\begin{align*}
\norm{\bar{f}(m)} = x^*(\bar{f}(m)) &= \sum_{k=1}^na_kx^*(\bar{f}(\delta(x_k))) \\
&= \sum_{k=1}^na_kx^*(f(x_k)) = \duality{x^*\circ f,m} \leq \norm{m}\Lip(x^*\circ f) \leq \norm{m}\Lip(f)
\end{align*}
and so $\bar{f}$ is a bounded linear operator with norm $\norm{\bar{f}}\leq\Lip(f)$. By density of $\lspan\,\delta(M)$, $\bar{f}$ can be extended to an operator $\lipfree{M}\to X$ with $\norm{\bar{f}}\leq\Lip(f)$. On the other hand, since $\delta$ is an isometry, we also have
$$
\norm{\bar{f}} = \sup_{m\in B_{\lipfree{M}}}\bar{f}(m) \geq \sup_{x\neq y\in M}\bar{f}\pare{\frac{\delta(x)-\delta(y)}{\norm{\delta(x)-\delta(y)}}} = \sup_{x\neq y\in M}\frac{f(x)-f(y)}{d(x,y)} = \Lip(f)
$$
hence $\norm{\bar{f}}=\Lip(f)$. Uniqueness of the extension $\bar{f}$ is obvious by linear density of $\delta(M)$.

To prove the uniqueness of $\lipfree{M}$, let $V$ be another Banach space and $e:M\to V$ an isometric embedding satisfying \ref{th:lipfree_universal:embedding} and \ref{th:lipfree_universal:linearization} in place of $\lipfree{M}$ and $\delta$, respectively. Applying \ref{th:lipfree_universal:linearization} to $e$ yields an operator $\bar{e}:\lipfree{M}\to V$ with $\bar{e}\circ\delta=e$ and $\norm{\bar{e}}=1$, and applying it to $\delta$ yields an operator $\bar{\delta}:V\to\lipfree{M}$ with $\bar{\delta}\circ e=\delta$ and $\norm{\bar{\delta}}=1$. For any $x\in M$ we have $\bar{\delta}\circ\bar{e}(\delta(x)) = \bar{\delta}(e(x)) = \delta(x)$ and $\bar{e}\circ\bar{\delta}(e(x)) = \bar{e}(\delta(x)) = e(x)$. Since $\delta(M)$ and $e(M)$ are linearly dense in $\lipfree{M}$ and $V$, respectively, $\bar{\delta}$ and $\bar{e}$ are inverse bijections, and since they are both non-expansive, they must be isometries.
\end{proof}

Part \ref{th:lipfree_universal:linearization} can be restated equivalently as follows. Let $\Lip_0(M;X)$ denote the space of $X$-valued Lipschitz functions on $M$ vanishing at $0$. Then $\Lip_0(M;X)$ is, for all purposes, identical to the space of bounded linear operators from $\lipfree{M}$ to $X$ (identifying $T:\lipfree{M}\to X$ with $T\circ\delta\in\Lip_0(M;X)$). Two particular cases of this fact yield important consequences. First, by taking $X=\RR$ it is immediate that $\lipfree{M}$, although defined as a subspace of the dual of $\Lip_0(M)$, is in fact a predual of $\Lip_0(M)$.

\begin{corollary}
\label{cr:lipfree_predual}
$\lipfree{M}^*$ is linearly isometric to $\Lip_0(M)$ under the mapping $\phi\mapsto\phi\circ\delta$. The corresponding weak$^\ast$ topology of $\Lip_0(M)$, when restricted to the unit ball $B_{\Lip_0(M)}$, coincides with the topology of pointwise convergence.
\end{corollary}

\begin{proof}
The first statement follows from Theorem \ref{th:lipfree_universal} by taking $X=\RR$.
For the second statement, let $(f_i)$ be a net in $B_{\Lip_0(M)}$. If $f_i\wsconv f$ then, in particular, $f_i(x)=\duality{f_i,\delta(x)}\to\duality{f,\delta(x)}=f(x)$ for all $x\in M$, that is, $f_i\to f$ pointwise. Thus the identity $(B_{\Lip_0(M)},w^*)\to (B_{\Lip_0(M)},\tau_p)$ (where $\tau_p$ is the topology of pointwise convergence) is a continuous bijection from a compact onto a Hausdorff space, and therefore a homeomorphism. This shows that the weak$^*$ and pointwise topologies agree on $B_{\Lip_0(M)}$.
\end{proof}


Another important particular case of Theorem \ref{th:lipfree_universal}, first highlighted by Godefroy and Kalton in \cite{GodefroyKalton}, gives us the following.

\begin{corollary}
\label{cr:linearization}
Let $f:M\to N$ be a Lipschitz function between metric spaces such that $f(0_M)=0_N$. Then there exists a (unique) bounded linear operator $\widehat{f}:\lipfree{M}\to\lipfree{N}$ that extends $f$ in the sense that $\widehat{f}\circ\delta_M=\delta_N\circ f$, and $\norm{\widehat{f}}=\Lip(f)$.
\end{corollary}

In other words, one can always complete the following commutative diagram:

$$
  \xymatrix{
    M \ar[d]_{\delta_M} \ar[rr]^f && N \ar[d]^{\delta_N} \\
	  \lipfree{M} \ar@{.>}[rr]_{\widehat{f}} && \lipfree{N}
  }
$$

\begin{proof}
Simply apply Theorem \ref{th:lipfree_universal} to $X=\lipfree{N}$ and the Lipschitz function $\delta_N\circ f:M\to\lipfree{N}$.
$$
\xymatrix{
M \ar[d]_{\delta_M} \ar[r]^f  & N \ar[r]^{\delta_N} & \lipfree{N} \\
\lipfree{M} \ar@{.>}[rru]_{\widehat{f} = \overline{\delta_N\circ f}}
}
$$
\end{proof}

The map $\widehat{f}$ is called the \emph{linearization} of $f$. It is easily checked that if $f$ is an isometry then so is $\widehat{f}$; more generally, if $f$ is a bi-Lipschitz mapping then $\widehat{f}$ is a linear isomorphism. Thus, bi-Lipschitz equivalent metric spaces have linearly isomorphic Lipschitz-free spaces.

\subsection{Lipschitz-free subspaces}

McShane's extension theorem guarantees that real-valued Lipschitz functions defined on a subset of $M$ can be extended to $M$ without increasing their Lipschitz constant (this is no longer true in general for $\CC$-valued, let alone vector-valued, Lipschitz functions). This admits the following equivalent statement in terms of Lipschitz-free spaces that plays a crucial role throughout most of the theory.

\begin{proposition}\label{pr:lipfree_subspace}
Let $N$ be a subset of $M$ containing $0$. Then $\lipfree{N}$ is linearly isometric to the subspace $\cl{\lspan}\,\delta(N)$ of $\lipfree{M}$.
\end{proposition}

\begin{proof}
Apply Theorem \ref{th:lipfree_universal} to the restricted mapping $\delta_M\restrict_N:N\rightarrow\lipfree{M}$ to obtain a linear operator $T:\lipfree{N}\rightarrow\lipfree{M}$ such that the following diagram is commutative:
$$
\xymatrix{
N \ar[d]_{\delta_N} \ar[r]^{\delta_M} & \lipfree{M} \\
\lipfree{N} \ar@{.>}[ur]_T &
}
$$
i.e.~such that $T(\delta_N(x))=\delta_M(x)$ for any $x\in N$. Now let $m\in\lspan\,\delta_N(N)\subset\lipfree{N}$, say $m=\sum_{i=1}^n a_i\delta_N(x_i)$ where $a_i\in\RR$ and $x_i\in N$. Then $T(m)=\sum_{i=1}^n a_i\delta_M(x_i)$, and we have
\begin{align*}
\norm{m}_{\lipfree{N}} &= \sup_{f\in B_{\Lip_0(N)}}\duality{f,m} = \sup_{f\in B_{\Lip_0(N)}}\sum_{i=1}^n a_if(x_i) \\
\norm{T(m)}_{\lipfree{M}} &= \sup_{f\in B_{\Lip_0(M)}}\duality{f,T(m)} = \sup_{f\in B_{\Lip_0(M)}}\sum_{i=1}^n a_if(x_i)
\end{align*}
By McShane's theorem, every function in $B_{\Lip_0(N)}$ can be extended to a function in $B_{\Lip_0(M)}$. Therefore both suprema have the same value, and $\norm{T(m)}=\norm{m}$. Thus $T$ restricted to $\lspan\,\delta_N(N)$ is an isometry, whose range is obviously $\lspan\,\delta_M(N)$, so it extends uniquely to an isometry (which must be $T$) between their completions $\lipfree{N}$ and $\cl{\lspan}\,\delta_M(N)$.
\end{proof}

In accordance with Proposition \ref{pr:lipfree_subspace}, we will always identify $\lipfree{N}$ with $\cl{\lspan}\,\delta(N)\subset\lipfree{M}$ whenever $0\in N\subset M$. Under this identification, several natural questions arise concerning the relationship between the different Lipschitz-free spaces over subsets $N\subset M$. In particular, is the class of Lipschitz-free subspaces closed under intersections? The answer is yes, which seems to be intuitively obvious, but it is actually nontrivial and was unknown until fairly recently. The proof, although elementary, is a big detour so we will omit it; the interested reader can find it in \cite{AP_supports} (for $M$ with finite diameter) and \cite{APPP} (extension to the general case).

\begin{theorem}[{\cite[Theorem 2.1]{APPP}}]
\label{th:intersection}
Let $M$ be a complete metric space and let $\set{N_i \,:\, i\in I}$ be a family of closed subsets of $M$ containing $0$. Then
$$
\bigcap_{i\in I}\lipfree{N_i} = \mathcal{F}\pare{\bigcap_{i\in I}N_i} .
$$
\end{theorem}

The most useful consequence of Theorem \ref{th:intersection} is that it allows us to define the \emph{support} of an element $m$ of $\lipfree{M}$ as the smallest closed set $N\subset M$ such that $m\in\lipfree{N\cup\set{0}}$; the theorem guarantees the existence of that set. See \cite{AP_supports,APPP} for more details.

The statement of Theorem \ref{th:intersection} assumes that $M$ and all the $N_i$ are complete metric spaces. The underlying fact (``intersections of Lipschitz-free spaces are Lipschitz-free spaces'') is still valid in the non-complete case, but the expression of the intersection becomes cumbersome. Note that, by its very definition, the Lipschitz space $\Lip_0(M)$ over any metric space agrees with the Lipschitz space $\Lip_0(\cl{M})$ over the completion $\cl{M}$ of $M$; the same holds for Lipschitz-free spaces, since they are defined entirely in terms of $\Lip_0(M)$. It is customary in Lipschitz-free space theory to focus on complete metric spaces, as it entails no loss of generality and allows for more elegant formulation of many results like Theorem \ref{th:intersection}.

\subsection{Examples in the real line}

We finish this section by describing the most important specific examples of Lipschitz-free spaces. When the metric space $M$ is finite, $\Lip_0(M)$ is obviously a finite-dimensional Banach space, so the same holds for $\lipfree{M}$. We will not be interested in these cases, focusing instead on infinite metric spaces $M$. Let us see how Lipschitz and Lipschitz-free spaces look like in the cases described in Example \ref{ex:examples_lip}.

\begin{example}
\label{ex:lipfree_N}
Let $M=\NN\cup\set{0}$ with the usual metric inherited from $\RR$. Then $\lipfree{M}$ is isometric to $\ell_1$ and $\Lip_0(M)$ is isometric to $\ell_\infty$. Indeed, let us see that the linear mapping $T:\Lip_0(M)\rightarrow\ell_\infty$ defined by $(Tf)_n=f(n)-f(n-1)$, $n\in\NN$ is a surjective isometry. Choose $f\in\Lip_0(M)$. It is clear that $\abs{f(n)-f(n-1)}\leq\Lip(f)$ for any $n\in\NN$, hence $\norm{Tf}_\infty\leq\Lip(f)$. On the other hand, for any $\varepsilon>0$ there are $m<n$ in $\NN\cup\set{0}$ such that
$$
\Lip(f)-\varepsilon<\frac{\abs{f(n)-f(m)}}{n-m}\leq\max_{m<k\leq n}\abs{f(k)-f(k-1)}
$$
and therefore $\norm{Tf}_\infty>\Lip(f)-\varepsilon$. Thus $T$ is an isometry. It is also easy to check that, for any $(a_n)\in\ell_\infty$, the function $f$ given by $f(0)=0$ and $f(n) = \sum_{k=1}^n a_k$ for $n\in\NN$ satisfies $\Lip(f)\leq\norm{(a_n)}_\infty$ and $Tf=(a_n)$. It follows that $T$ is surjective. Moreover, $T$ is clearly pointwise-to-pointwise continuous, hence (by Corollary \ref{cr:lipfree_predual} and the Banach-Dieudonn\'e theorem) an adjoint operator. The inverse of the preadjoint of $T$ provides an explicit isometry $S:\lipfree{M}\rightarrow\ell_1$ that maps $\delta(n)$ to the sequence $(1,\stackrel{n}{\ldots},1,0,0,\ldots)$.
\end{example}

\begin{example}
\label{ex:lipfree_R}
Let $M=[0,1]$ with the usual metric. Let us see that $\Lip_0(M)$ is isometric to $L^\infty([0,1])$. Consider the linear operator $T_1:L^\infty([0,1])\rightarrow\Lip_0(M)$ given by
$$
(T_1f)(x)=\int_0^x f(t)\,dt .
$$
It is clear that $\Lip(T_1f)\leq\norm{f}_\infty$, so $T_1$ is non-expansive. On the other hand, consider the operator $T_2:\Lip_0([0,1])\rightarrow L^\infty([0,1])$ given by $T_2f=f'$. By Rademacher's theorem, $T_2$ is well-defined, also non-expansive, and $T_1$ and $T_2$ are inverses of each other by the fundamental theorem of calculus. It follows that both $T_1$ and $T_2$ are surjective linear isometries. Moreover, if $(f_i)$ is a net in $L^\infty([0,1])$ with weak$^*$ limit $f$, then
$$
\lim_i (T_1f_i)(x) = \lim_i \int_0^x f_i(t)\,dt = \int_0^x f(t)\,dt = (T_1f)(x)
$$
for any $x\in [0,1]$. Thus $T_1$ is weak$^*$-to-pointwise, hence (by Corollary \ref{cr:lipfree_predual}) weak$^*$-to-weak$^*$ continuous. Its preadjoint is a linear isometry $\lipfree{M}\to L^1([0,1])$ that takes $\delta(x)$ to $\chi_{[0,x]}$.
\end{example}

Godard proved in \cite{Godard} that these two examples essentially cover all possible subsets $M$ of $\RR$ (and, in fact, one can find explicit linear isometries for any $M$ that do not look very different from the ones given in Examples \ref{ex:lipfree_N} and \ref{ex:lipfree_R}).

\begin{theorem}[{cf.~\cite[Corollary 3.4]{Godard}}]
\label{th:godard}
Let $M\subset\RR$ be closed and infinite. Then:
\begin{enumerate}[label={\upshape{(\alph*)}}]
\item If $\HH^1(M)=0$ then $\lipfree{M}$ is isometric to $\ell_1$.
\item If $\HH^1(M)>0$ then $\lipfree{M}$ is isomorphic to $L^1([0,1])$.
\end{enumerate}
\end{theorem}

We finish this section by remarking that, while the situation is completely understood for subsets of the real line, our knowledge is hopelessly limited in higher dimensions. A celebrated result of Naor and Schechtman \cite{NaorSchechtman} states that $\lipfree{\RR^2}$ cannot be isomorphically embedded into $L^1([0,1])$, so Proposition \ref{pr:lipfree_subspace} implies that $\lipfree{\RR^n}$ cannot be isomorphic to (a subspace of) $\lipfree{\RR}$ for $n>1$. Whether $\lipfree{\RR^m}$ and $\lipfree{\RR^n}$ are isomorphic for any $n>m\geq 2$ is an open problem at the time of writing, perhaps the most important one in Lipschitz-free space theory.

\section{Equivalence - the compact case}
\label{sec:compact}

We now proceed to describe several functional analytic properties of Lipschitz-free spaces that are equivalent to pure $1$-unrectifiability of the underlying metric space, under the assumption that the metric space is compact. In the next section, we will see which of these equivalences carry over to the general case. We start by stating the equivalence theorem before defining the new notation appearing in it.

\begin{theorem}[{\cite[Theorem 3.1]{AGPP}}]
\label{th:compact}
Let $M$ be a compact metric space. Then the following are equivalent:
\begin{enumerate}[label={\upshape{(\roman*)}}]
\item $M$ is purely $1$-unrectifiable,
\item $\lip(M)$ separates points of $M$ uniformly,
\item $\lip_0(M)^*=\lipfree{M}$,
\item $\lipfree{M}$ is a dual Banach space,
\item $\lipfree{M}$ has the Radon-Nikod\'ym property,
\item $\lipfree{M}$ has the Schur property,
\item $\lipfree{M}$ does not have a subspace isomorphic to $L^1([0,1])$.
\end{enumerate}
\end{theorem}

We will now go over each implication one by one, providing any missing definitions as needed. We will not provide full proofs of every implication; instead, we will only sketch the ideas behind some of them.

\medskip
\textbf{Implication (i)$\Rightarrow$(ii):}
This is given by Theorem \ref{th:agpp spu}. Although we have covered it here first, it was the last implication to be proved historically, and the one that completed the theorem in \cite{AGPP}. Recall that, by Proposition \ref{pr:spu p1u}, we also have (ii)$\Rightarrow$(i).

\medskip
\textbf{Implications (ii)$\Rightarrow$(iii)$\Rightarrow$(iv):}
Similarly to how Lipschitz spaces were defined in Section \ref{sec:lipfree}, one can define the restricted space of locally flat Lipschitz functions
$$
\lip_0(M) = \set{f\in\Lip_0(M) \,:\, \text{$f$ is locally flat}} .
$$
This is often called the \emph{little Lipschitz space} over $M$. One of the properties stated after Definition \ref{def:locally flat} above is that norm limits of locally flat functions are locally flat; it follows that $\lip_0(M)$ is a closed subspace of $\Lip_0(M)$, hence a Banach space with the Lipschitz norm as well.

\begin{remark}
Note that, unlike Lipschitz spaces, little Lipschitz spaces are \textit{not} invariant under taking the completion of the metric space. For instance, $\Lip_0(\QQ)$ can be identified with $\Lip_0(\RR)$, as any Lipschitz function on $\QQ$ admits a unique continuous extension to $\RR$ with the same Lipschitz constant. However, $\lip_0(\QQ)$ is very different from $\lip_0(\RR)$. Indeed, $\lip_0(\RR)$ denotes the set of functions in $\Lip_0(\RR)$ that are locally flat at every point of $\RR$, which is just $\set{0}$ by Example \ref{ex:examples_lip}\ref{ex:lipR}. On the other hand, functions in $\lip_0(\QQ)$ are only required to be locally flat at every point of $\QQ$; in fact, $\lip_0(\QQ)$ can be identified with the space of functions in $\Lip_0(\RR)$ that are locally flat at every point of $\QQ$. Since $\QQ$ is countable, it is possible for such a function to be non-trivial - in fact, $\lip_0(\QQ)$ is quite large by Example \ref{ex:examples_lip}\ref{ex:real_null}.
\end{remark}

It should be clear that $\lip(M)$ separates points of $M$ uniformly if and only if $\lip_0(M)$ does (that is, one may replace $\lip(M)$ with $\lip_0(M)$ in Definition \ref{def:spu}), as that property is invariant under adding a constant to functions. Similarly, in Proposition \ref{pr:spu eqv density} one may replace $\Lip(M)$ and $\lip(M)$ with $\Lip_0(M)$ and $\lip_0(M)$, respectively. Taking Corollary \ref{cr:lipfree_predual} into account, this yields the following equivalence:

\begin{proposition}
\label{pr:spu eqv goldstine}
Let $M$ be a metric space. Then $\lip(M)$ separates points of $M$ uniformly if and only if $B_{\lip_0(M)}$ is weak$^*$ dense in $B_{\Lip_0(M)}$.
\end{proposition}

Note that Proposition \ref{pr:spu eqv goldstine} provides the implication (iii)$\Rightarrow$(ii). Indeed, by Goldstine's theorem $B_{\lip_0(M)}$ is weak$^*$ dense in $B_{\lip_0(M)^{**}}$, and assumption (iii) entails that $\lip_0(M)^{**}=\lipfree{M}^*=\Lip_0(M)$. Proving the converse implication requires a method for recognizing possible preduals of a Banach space $X$ inside of its dual $X^*$. There are several known preduality criteria; the following one, due to Petunin and Plichko \cite{PetuninPlichko}, sees frequent use in Lipschitz-free space theory.

\begin{theorem}[{\cite[Theorem 4]{PetuninPlichko}}]
\label{th:petunin plichko}
Let $X$ be a separable Banach space. Then a closed subspace $Y$ of $X^*$ is a predual of $X$ if and only if it is weak$^*$ dense in $X^*$ and every element of $Y$ attains its norm on $X$.
\end{theorem}

The last condition means the following: a functional $x^*\in X^*$ attains its norm if there exists $x\in X$, $x\neq 0$ such that $\duality{x^*,x}=\norm{x^*}\norm{x}$; that is, if the supremum defining its norm $\norm{x^*}$ is a maximum.

One drawback of this criterion is that it requires separability. Thankfully, this is not an issue in our case: any compact metric space $M$ is separable and thus $\lipfree{M}$, being linearly generated by the separable set $\delta(M)$, is separable as well. 

\begin{theorem}\label{th:spu dual}
Let $M$ be a compact metric space. If $\lip(M)$ separates points of $M$ uniformly, then $\lip_0(M)^*=\lipfree{M}$.
\end{theorem}

\begin{proof}
We apply Theorem \ref{th:petunin plichko} to $X=\lipfree{M}$ and $Y=\lip_0(M)$. We have already mentioned that $\lip_0(M)$ is a closed subspace of $\Lip_0(M)$, and weak$^*$ density is given by Proposition \ref{pr:spu eqv goldstine}, so we only need to show that every function $f\in\lip_0(M)$ attains its norm on some element of $\lipfree{M}$. Let $f\in\lip_0(M)$ and assume $\Lip(f)>0$. By the definition of Lipschitz constant, there exist sequences of points $(x_n)$ and $(y_n)$ in $M$ such that $x_n\neq y_n$ and
$$
\Lip(f) - \frac{1}{n} < \frac{f(x_n)-f(y_n)}{d(x_n,y_n)} \leq \Lip(f)  .
$$
Since $M$ is compact, we may pass to subsequences such that $x_n\to x\in M$, $y_n\to y\in M$, and
$$
\lim_{n\to\infty} \frac{f(x_n)-f(y_n)}{d(x_n,y_n)} = \Lip(f) .
$$
Note that $x$ and $y$ must be different: otherwise, local flatness at $x=y$ would imply that the limit above is $0$, a contradiction. Then, by continuity
$$
\Lip(f) = \frac{f(x)-f(y)}{d(x,y)} = \duality{f,\frac{\delta(x)-\delta(y)}{d(x,y)}}
$$
so $f$ attains its norm at the element $(\delta(x)-\delta(y))/d(x,y)$ of $\lipfree{M}$, of norm $1$.
\end{proof}

The implication (ii)$\Rightarrow$(iii) is thus proved, and (iii)$\Rightarrow$(iv) is of course trivial.

The original proof of Theorem \ref{th:spu dual} was essentially due to Hanin in \cite{Hanin} and refined to its current form by Weaver in \cite{Weaver96_1}; it can also be found in \cite[Chapter 4]{Weaver2}. The proof provided here is different, much more similar to the arguments used in \cite{Kalton} and \cite{GPPR} for related results.

\medskip
\textbf{Implication (iv)$\Rightarrow$(v):}
This is given by the following classical result of Dunford and Pettis; see \cite[Theorem 5.5.6]{AlbiacKalton} for a simple proof.

\begin{theorem}
\label{th:separable dual rnp}
Every separable dual Banach space has the Radon-Nikod\'ym property.
\end{theorem}

\medskip
\textbf{Implication (ii)$\Rightarrow$(vi):}
The relation between duality and the Schur property in Lipschitz-free spaces was first studied by Kalton in \cite{Kalton}, where he proved that $\lipfree{M}$ has the Schur property whenever $M$ is a snowflake. Later, in \cite{Petitjean}, Petitjean analyzed Kalton's proof and found that, in fact, the only real requirement in Kalton's argument was that the uniformly locally flat $1$-Lipschitz functions on $M$ are weak$^*$ dense in $B_{\Lip_0(M)}$ (even more, one only really needs density of the set of functions with uniformly small asymptotic Lipschitz constant, like those considered by Bate in \cite{Bate}). 

\begin{theorem}[{\cite[Proposition 8]{Petitjean}}]
\label{th:spu schur}
Let $M$ be a metric space. If the uniformly locally flat Lipschitz functions on $M$ separate points of $M$ uniformly, then $\lipfree{M}$ has the Schur property.
\end{theorem}

Since local flatness and uniform local flatness are equivalent in the compact setting, this theorem provides the implication (ii)$\Rightarrow$(vi). While its proof is not too difficult, it is based on a somewhat lengthy construction, so we shall omit it here and refer the interested reader to either \cite{Kalton} or \cite{Petitjean} instead. Notably, the proof is based on a lemma by Kalton that will be of significance in Section \ref{sec:non-compact} (see Lemma \ref{lm:kalton} below).

\medskip
\textbf{Implications (v)$\Rightarrow$(vii) and (vi)$\Rightarrow$(vii):}
It follows immediately from the definition of the Schur and the Radon-Nikod\'ym properties that both are invariant under isomorphisms and under passing to subspaces. While $\ell_1$ is the archetype example of a Banach space satisfying both, the paradigmatic examples of spaces failing both properties are $c_0$ and $L^1([0,1])$; the latter provides us with the desired implications. Indeed, it is easy to check that the mapping $f:[0,1]\to L^1([0,1])$ given by $f(t)=\chi_{[0,t]}$ is an isometry that is nowhere differentiable, so $L^1([0,1])$ fails Definition \ref{def:rnp}(a). On the other hand, the functions $f_n(t)=\sin(n\pi t)$ define a sequence in $L^1([0,1])$ that converges weakly to $0$ by the classical Riemann-Lebesgue lemma \cite[5.14]{Rudin}, but not in norm as $\norm{f_n}_1=\frac{2}{\pi}$ for all $n$, so $L^1([0,1])$ fails Definition \ref{def:rnp}(b) as well.

We remark here that, contrary to what these examples would suggest, neither the Schur nor the Radon-Nikod\'ym property implies the other one in the general Banach space setting.

\medskip
\textbf{Implication (vii)$\Rightarrow$(i):}
The remaining implication is easy and follows from Godard's theorem and the basic linearization properties of Lipschitz-free spaces.

\begin{proposition}
Let $M$ be a metric space. Suppose that $M$ is not purely $1$-unrectifiable. Then $\lipfree{M}$ contains an isomorphic copy of $L^1([0,1])$.
\end{proposition}

\begin{proof}
If $M$ is not purely $1$-unrectifiable then it contains a curve fragment, i.e. there exist a compact set $K\subset\RR$ with $\HH^1(K)>0$ and a bi-Lipschitz embedding $\varphi:K\to M$. Then $\lipfree{\varphi(K)}$ is a subspace of $\lipfree{M}$ that is isomorphic to $\lipfree{K}$ (via the linearization $\widehat{\varphi}$), which is in turn isomorphic to $L^1([0,1])$ by Theorem \ref{th:godard}.
\end{proof}

\section{Extension to the non-compact case}
\label{sec:non-compact}

This final section will focus on the possible extension of the equivalences in Theorem \ref{th:compact} to the general case where $M$ is no longer compact; we will assume, instead, that $M$ is complete. Our approach for extension will be to reduce the problem to the compact case by showing that some of the properties appearing in Theorem \ref{th:compact} are \emph{compactly determined}, meaning that they hold for $M$ as soon as they hold for all compact subsets of $M$. Note that pure $1$-unrectifiability \textit{is} compactly determined as it is characterized by the non-containment of curve fragments, which are compact sets. Thus, any property of $M$ that is equivalent to pure $1$-unrectifiability must also be compactly determined.

\subsection{Compact reduction}

The method of compact reduction for Lipschitz-free spaces was introduced in \cite{ANPP} by Petitjean and Proch\'azka together with the author (the other named author of \cite{ANPP}, C. No\^us, is actually a pseudonym). The starting point is a lemma by Kalton that was originally used to prove a restricted version of Theorem \ref{th:spu schur}. Kalton's lemma states, informally, that a weakly convergent sequence in a Lipschitz-free space $\lipfree{M}$ must be ``almost supported on a small subset of $M$''. By this, Kalton means that the sequence can be approximated by elements in the Lipschitz-free space over a small subset of $M$, meaning a set that can be covered by finitely many balls of small radius (that is, a small set in the sense of Kuratowski's measure of non-compactness). Using the notation
$$
[E]_r = \set{x\in M \,:\, d(x,E)\leq r} ,
$$
we consider small sets of the form $[E]_r$ where $E\subset M$ is finite and $r>0$ is small. Kalton's lemma then claims the following.

\begin{lemma}[{\cite[Lemma 4.5]{Kalton}}]
\label{lm:kalton}
Let $M$ be a metric space with finite diameter, $(w_n)$ be a weakly convergent sequence in $\lipfree{M}$, and $W=\set{w_n : n\in\NN}$. Then, for every $\ep>0$ and $r>0$ there exist a finite set $E\subset M$ and a mapping $T:W\to\lipfree{[E]_r}$ such that $\norm{w-T(w)}\leq\ep$ for all $w\in W$.
\end{lemma}

In \cite{ANPP}, Kalton's lemma was improved in two ways. The first was to extend its range of application from weakly convergent sequences to more general sets, like weakly compact sets, as well as removing the constraint that $M$ has finite diameter.

\begin{proposition}[{\cite[Proposition 3.6]{ANPP}}]
\label{pr:kalton premium}
Let $M$ be a metric space and $W$ be a weakly compact subset of $\lipfree{M}$. Then, for every $\ep>0$ and $r>0$ there exist a finite set $E\subset M$ and a mapping $T:W\to\lipfree{[E]_r}$ such that $\norm{w-T(w)}\leq\ep$ for all $w\in W$.
\end{proposition}

The proof of Proposition \ref{pr:kalton premium} is a straightforward adaptation of Kalton's original and will be omitted here. In fact, Kalton's lemma holds for more general classes of subsets of $\lipfree{M}$, like weakly precompact sets and (V*)-sets; see \cite[Theorem 2.7]{APQ}. We will not define these classes of sets as they are not needed here.

The second improvement to Kalton's lemma concerns its conclusion. Namely, one can always substitute the ``small sets'' with more manageable \textit{compact} sets (assuming completeness of $M$, otherwise one may need to use totally bounded sets instead). Moreover, the approximation mapping $T$ can be chosen to be linear and continuous.

\begin{proposition}[{\cite[Theorem 3.2]{ANPP}}]
\label{pr:compact reduction}
Let $M$ be a complete metric space. For a subset $W\subset\lipfree{M}$, the following are equivalent:
\begin{enumerate}[label={\upshape{(\roman*)}}]
\item For every $\ep>0$ and $r>0$ there exist a finite set $E\subset M$ and a mapping $T:W\to\lipfree{[E]_r}$ such that $\norm{w-T(w)}\leq\ep$ for all $w\in W$.
\item For every $\ep>0$ there exist a compact set $K\subset M$ and a mapping $T:W\to\lipfree{K}$ such that $\norm{w-T(w)}\leq\ep$ for all $w\in W$.
\end{enumerate}
Moreover, the mapping $T$ can be chosen to be continuous and linear. More precisely, there exists a sequence of bounded linear operators $T_n:\lipfree{M}\to\lipfree{M}$ that converge to $T$ uniformly on $W$.
\end{proposition}

We clarify here that the approximation mapping $T$ is only defined on $W$, not on the whole space $\lipfree{M}$, and does not extend to a continuous linear operator on $\lipfree{M}$ in general. Similarly, $T_n(x)$ is only guaranteed to converge for $x\in W$, and the norms of the operators $T_n$ will not be bounded uniformly in general.

We only sketch the main idea of the proof of Proposition \ref{pr:compact reduction}. Of course (ii) implies (i) as any compact set $K$ can be covered with finitely many balls of any given radius, so one only has to prove that (i) implies (ii). Let $(r_n)$, $(\ep_n)$ be sequences of positive real numbers decreasing to $0$, with $\sum_n\ep_n\leq\ep$. By assumption, there are a finite set $E_1\subset M$ and an approximation mapping $T_1:W\to\lipfree{[E_1]_{r_1}}$ with $\norm{w-T_1(w)}\leq\ep_1$ for all $w\in W$. It can be shown that the image $W_1=T_1(W)$ satisfies (i) as well \textit{in the Lipschitz-free subspace $\lipfree{[E_1]_{r_1}}$}, so one can find a finite set $E_2\subset [E_1]_{r_1}$ and an approximation mapping $T_2:W_1\to\lipfree{[E_1]_{r_1}\cap[E_2]_{r_2}}$ with $\norm{w-T_2(w)}\leq\ep_2$ for all $w\in W_1$. Set $W_2=T_2(W_1)$. Now one continues iteratively, obtaining finite sets $E_n$ and mappings
$$
T_n:W_{n-1}\to\lipfree{[E_1]_{r_1}\cap[E_2]_{r_2}\cap\ldots\cap [E_n]_{r_n}}
$$
with error $\ep_n$ and image $W_n$. By choosing the constants $r_n$ and $\ep_n$ appropriately, one can guarantee that the iterations are always possible and the approximations converge to some set that will be contained in
$$
\bigcap_{n=1}^\infty \lipfree{[E_1]_{r_1}\cap[E_2]_{r_2}\cap\ldots\cap [E_n]_{r_n}} = \mathcal{F}\pare{\bigcap_{n=1}^\infty [E_n]_{r_n}} ,
$$
where we crucially use Theorem \ref{th:intersection}. The set $\bigcap_{n=1}^\infty [E_n]_{r_n}$ is then the sought compact $K$: it is an intersection of closed sets and also totally bounded as, for any $r>0$, it can be covered with
$$
[E_n]_r = \bigcup_{x\in E_n} B(x,r)
$$
whenever $r_n\leq r$. In the construction, one may even arrange for the mappings $T_n$ to be bounded linear operators defined on the whole space $\lipfree{[E_1]_{r_1}\cap\ldots\cap [E_{n-1}]_{r_{n-1}}}$; the operators $T_n$ in the statement of Proposition \ref{pr:compact reduction} are then the compositions $T_n\circ\ldots\circ T_2\circ T_1$ in the notation of our sketched proof.

The combination of Propositions \ref{pr:kalton premium} and \ref{pr:compact reduction} allows us to prove rather easily that many Banach space properties of Lipschitz-free spaces are compactly determined, as long as they are somehow related to the weak topology. One such example is the Schur property.

\begin{theorem}[{\cite[Corollary 2.7]{ANPP}}]
\label{th:schur cd}
Let $M$ be a complete metric space. Then $\lipfree{M}$ has the Schur property if and only if $\lipfree{K}$ has the Schur property for every compact set $K\subset M$.
\end{theorem}

\begin{proof}
The forward implication is obvious as the Schur property passes to subspaces. For the converse, let $(w_n)$ be a sequence in $\lipfree{M}$ that converges weakly to $0$; our goal is to show that $\norm{w_n}\to 0$. Fix $\ep>0$. By Propositions \ref{pr:kalton premium} and \ref{pr:compact reduction}, we can find a compact set $K\subset M$, approximations $T(w_n)\in\lipfree{K}$ of $w_n$ such that $\norm{w_n-T(w_n)}\leq\ep$, and a sequence $(T_k)$ of continuous linear operators on $\lipfree{M}$ such that $T_k(w_n)$ converge to $T(w_n)$ uniformly on $n$.

We claim that $T(w_n)$ converges weakly to $0$ as well. Indeed, fix $\delta>0$ and $f\in\Lip_0(M)$, $f\neq 0$, and let us check that $\abs{\duality{f,T(w_n)}}\leq 2\delta$ when $n$ is large enough. First, by uniform convergence we may find $k\in\NN$ such that $\norm{T_k(w_n)-T(w_n)}\leq\delta/\Lip(f)$ for all $n$. Since $T_k$ is a continuous operator, it is continuous with respect to the weak topology as well, and so $T_k(w_n)$ converges weakly to $0$; thus, we have $\abs{\duality{f,T_k(w_n)}}\leq\delta$ for $n$ large enough. For such $n$, we get
$$
\abs{\duality{f,T(w_n)}} \leq \abs{\duality{f,T(w_n)-T_k(w_n)}} + \abs{\duality{f,T_k(w_n)}} \leq \Lip(f)\cdot\norm{T(w_n)-T_k(w_n)} + \delta \leq 2\delta .
$$
This proves our claim.

By assumption, $\lipfree{K}$ has the Schur property, so we conclude that $\norm{T(w_n)}\to 0$. Thus, for $n$ large enough we have $\norm{T(w_n)}\leq\ep$ and $\norm{w_n}\leq\norm{w_n-T(w_n)}+\norm{T(w_n)}\leq 2\ep$. Therefore $\norm{w_n}\to 0$ as required.
\end{proof}

Section 2 of \cite{ANPP} contains other examples of compactly determined properties of $\lipfree{M}$, such as weak sequential completeness, the approximation property, or the Dunford-Pettis property. In all cases, the proof follows a similar scheme: approximate a certain set in a subspace $\lipfree{K}$ using Propositions \ref{pr:kalton premium} and \ref{pr:compact reduction}, apply the hypothesis on $\lipfree{K}$ to show that the approximated object possesses the desired property, and finally lift that property to the original set.

\subsection{Probabilistic compact reduction}
\label{sec:probabilistic}

Unfortunately, the method presented above is not enough to establish compact determination of the Radon-Nikod\'ym property when used as is. This is possibly due to the fact that there is no satisfactory characterization of the Radon-Nikod\'ym property in terms of weak topological properties of certain subsets. Instead, it is possible to develop a probabilistic version of the method, more suitable to the characterizations of the Radon-Nikod\'ym property in terms of integration. This modification can be found in Section 4 of \cite{AGPP}. The core arguments remain the same but the technical complexity is increased due to the different context, so we only give an overall description. We recommend \cite[Chapters 1 and 2]{Pisier} or \cite[Chapter 5]{BenyaminiLindenstrauss} for reference on this probabilistic setting.

The main idea is to replace (sets of) \textit{elements of $\lipfree{M}$} in Propositions \ref{pr:kalton premium} and \ref{pr:compact reduction} with (sets of) \textit{random variables taking values in $\lipfree{M}$}, that is, elements of the Bochner space $L^1(\mu;\lipfree{M})$ of $\lipfree{M}$-valued integrable functions on a given probability measure space. The resulting probabilistic version of Proposition \ref{pr:compact reduction} is valid with no further change (see \cite[Proposition 4.2]{AGPP}). For Proposition \ref{pr:kalton premium}, we replace the weakly convergent sequence with a martingale in $L^1(\mu;\lipfree{M})$ (see \cite[Proposition 4.3]{AGPP}). Finally, we use the following characterization of the Radon-Nikod\'ym property:
\begin{itemize}
\item A Banach space $X$ has the Radon-Nikod\'ym property if and only if every uniformly integrable martingale in $L^1(\mu;X)$ converges in norm (see \cite[Theorem 2.9]{Pisier}).
\end{itemize}
Based on it, an argument similar to that of Theorem \ref{th:schur cd} can prove compact determination of the Radon-Nikod\'ym property.



\begin{theorem}[{\cite[Corollary 4.5]{AGPP}}]
\label{th:rnp cd}
Let $M$ be a complete metric space. Then $\lipfree{M}$ has the Radon-Nikod\'ym property if and only if $\lipfree{K}$ has the Radon-Nikod\'ym property for every compact set $K\subset M$.
\end{theorem}

\subsection{Equivalences in the general case}

Having succeeded in establishing that some of the properties in Theorem \ref{th:compact} are compactly determined, we can extend the equivalence between them to the non-compact case. We keep the same numbering for the properties as in Theorem \ref{th:compact}.

\begin{theorem}[{\cite[Theorem C]{AGPP}}]
\label{th:non-compact}
Let $M$ be a complete metric space. Then the following are equivalent:
\begin{enumerate}[label={\upshape{(\roman*)}}]
\item $M$ is purely $1$-unrectifiable,
\setcounter{enumi}{4} 
\item $\lipfree{M}$ has the Radon-Nikod\'ym property,
\item $\lipfree{M}$ has the Schur property,
\item $\lipfree{M}$ does not contain isomorphic copies of $L^1([0,1])$.
\end{enumerate}
\end{theorem}

\begin{proof}
Properties (i), (v) and (vi) are compactly determined by Theorems \ref{th:p1u kirchheim}, \ref{th:rnp cd} and \ref{th:schur cd}, respectively, and they are equivalent to each other for compact $M$ by Theorem \ref{th:compact}, therefore they must be equivalent in general. The implications (v)$\Rightarrow$(vii), (vi)$\Rightarrow$(vii) and (vii)$\Rightarrow$(i) hold even in the non-compact case, as their proofs given in Section \ref{sec:compact} did not require compactness.
\end{proof}

As a corollary of Theorem \ref{th:non-compact}, we deduce that property (vii) is also compactly determined, although we have no direct proof of that fact.

The argument for equivalence with (vii) can be replicated with some other properties of $\lipfree{M}$ in order to add them to Theorem \ref{th:non-compact}. The most notable one is probably the \emph{Krein-Milman property}, which is satisfied by a Banach space $X$ whenever every closed bounded and convex subset of $X$ equals the closed convex hull of its extreme points. Indeed, that property is implied by (v) (see the equivalences after Definition \ref{def:rnp}) and, in turn, implies (vii); thus, the Krein-Milman property and the Radon-Nikod\'ym property are equivalent for Lipschitz-free spaces. Whether these two properties are equivalent in general is probably the most important open problem in the geometry of Banach spaces.

Of course, one may wonder whether the remaining properties (ii)-(iv) appearing in Theorem \ref{th:compact}, which are equivalent in the compact case, are still equivalent in general. Condition (ii), i.e.~``$\lip(M)$ separates points of $M$ uniformly'', always implies the others by Proposition \ref{pr:spu p1u}, but the converse implication is currently an open problem; see Question \ref{q:spu p1u}. On the other hand, conditions (iii) and (iv) involving the duality of $\lipfree{M}$ cannot be equivalent to the others in general. Indeed, Example 5.8 in \cite{GPPR} describes a simple complete, countable metric space (hence purely $1$-unrectifiable) whose Lipschitz-free space admits no predual, so (i) does not imply (iii) or (iv) in general. Conversely, duality implies (v) when $M$ is separable by Theorem \ref{th:separable dual rnp}, but not in the non-separable case as witnessed by \cite[Corollary 3.2]{AGP}: one can construct a metric space $M$ whose Lipschitz-free space is isometric to the non-separable dual $C([0,1])^*$, which lacks the Radon-Nikod\'ym and Schur properties.

\section*{Acknowledgments}

The contents of this paper correspond approximately to a lecture delivered by the author at the Second Winter School in Geometric Measure Theory at Westlake University, Hangzhou, on February 3rd, 2026. The author wishes to thank Thierry de Pauw and the rest of the organizing team for their kind invitation, warm welcome and excellent organization. He also wishes to thank the anonymous referee of this manuscript for their careful reading and error detection.


\end{document}